\documentclass{article} 
\usepackage{amsmath,amssymb,amsthm,amsfonts}
\usepackage{enumerate,xcolor}
\usepackage[a4paper]{geometry}
\usepackage[utf8]{inputenc}
\usepackage[T1]{fontenc}
\usepackage{lmodern}
\usepackage{graphicx}
\usepackage{placeins}% admits FloatBarrier, used to prevent the Figures mingling with the bibliography
\usepackage{float} %adds option[H] for figure placement, meaning ``exactly here''

\makeatletter
\renewcommand{\section}{%
\@startsection{section}{1}{\z@}
{0.5truecm plus -1ex minus -.2ex}%
{1.0ex plus .2ex}{\bfseries\large}}
\def\@seccntformat#1{\csname the#1\endcsname.\ }
\makeatother

\numberwithin{equation}{section} 
\newtheorem{thm}{Theorem}[section]

\newtheorem{lem}[thm]{Lemma}

\theoremstyle{definition}

\newcommand{\ep}{\varepsilon}
\newcommand{\pa}{\partial}

\newcommand{\Rone}{\mathbb{R}}

\newcommand{\utilde}{\widetilde{u}}
\newcommand{\vtilde}{\widetilde{v}}

\newcommand{\ol}[1]{\overline{#1}}
\newcommand{\ul}[1]{\underline{#1}}
\newcommand{\tmax}{T_{\rm max}}
\newcommand{\lp}[2]{\Vert{#2}\Vert_{L^{#1}(\Omega)}}
\newcommand{\Lone}{\mathcal{L}_1}
\newcommand{\Ltwo}{\mathcal{L}_2}

\definecolor{blue}{HTML}{0020A1}

\newcommand{\nn}{\nonumber}

\title{\mbox{On the weakly competitive case in a two-species chemotaxis model}}
\author{Tobias Black\thanks{Universit\"at Paderborn, Institut f\"ur Mathematik, Warburger Str. 100, 33098 Paderborn, Germany; tblack@math.uni-paderborn.de}\quad and Johannes Lankeit\thanks{Universit\"at Paderborn, Institut f\"ur Mathematik, Warburger Str. 100, 33098 Paderborn, Germany; jlankeit@math.uni-paderborn.de}\quad and Masaaki Mizukami\thanks{Department of Mathematics, Tokyo University of Science; 
masaaki.mizukami.math@gmail.com}}
\date{}

\begin{document}
%==========================title==========================
\maketitle 
%=====================  Abstract  =======================
\begin{abstract}
\noindent{%\footnotesize %{\bf Abstract.}
 In this article we investigate a parabolic-parabolic-elliptic two-species chemotaxis system with weak competition and show global asymptotic stability of the coexistence steady state under a smallness condition on the chemotactic strengths, which seems more natural than the condition previously known. \\
 For the proof we rely on the method of eventual comparison, which thereby is shown to be a useful tool even in the presence of chemotactic terms.
    
    % 35M33 (2010-now) Initial-boundary value problems for systems of mixed type 
    % 35A01 (2010-now) Existence problems: global existence, local existence, non-existence 
    % 35B35 (1973-now) Stability 
    % 35B51 (2010-now) Comparison principles 
    %  35K15 (1973-now) Initial value problems for second-order parabolic equations 
    % 35K55 (1973-now) Nonlinear parabolic equations 
    % 35B40 (1973-now) Asymptotic behavior of solutions 
    % 35K20 (1973-now) Initial-boundary value problems for second-order parabolic equations 
    % 92C17 (2000-now) Cell movement (chemotaxis, etc.) 
    % 92D40 (1991-now) Ecology 
}
\end{abstract}
%\vspace{10pt}

%%==============================================================%%
%%==============                                  ==============%%
%%======                      Section1                    ======%%
%%====                                                      ====%%
%%==                                                          ==%%
%%====                      Introduction                    ====%%
%%======                                                  ======%%
%%==============                                  ==============%%
%%==============================================================%%

\section{Introduction}
The questions how different populations living in the same habitat interact with each other and with their surroundings are central to mathematical biology. 

The competition for resources by two species (in contrast to, e.g., one being prey to the other) is often modeled by means of 
%If two species mainly interact by competition for resources (in contrast to, e.g.\red{,} one being prey to the other), 
%this is perhaps \red{most easily} included into a model by means of 
Lotka-Volterra type competition terms, i.e. with coupling coefficients $a_1>0$, $a_2>0$ in 
\[
 u_t=u(1-u-a_1v), \qquad v_t=v(1-v-a_2u). 
\]
We refer to \cite{goh_twospecies} for conditions for global stability of fixed points in such systems: If $a_1>0$ and $a_2>0$, for positive equilibria to be globally stable it is sufficient that they be locally stable.

Of course, spatial homogeneity is a highly idealized situation, and dependence of the population size 
on a spatial variable together with the effects of random motion of individuals has been incorporated into the model (see e.g. \cite[Chapter 1.2]{murrayII} and references therein).
% For example the 
% stability of the steady states 
% (that is, of coexistence equilibria and of those steady states corresponding to extinction of one species) 
% 
% has been studied extensively, e.g. in  
% 
% to name but a few of the references. 

One way of even simplest lifeforms to react to their environment
is chemotaxis, that is the tendency to move in the direction of higher concentrations of a signal substance.
% Chemotactic effects (or e.g. prey-tactic, which mathematically take the same form) play an essential role in numerous biological context, not only ranging from the aggregation of 
% 
% to a multicellular organism in Dictyostelium discoideum 
% 
% %to the invasion of tomato stems by bacteria \cite{yao_allen}, wound healing or  .
Inter alia, effects of chemotaxis on possible population size have been considered in \cite{lauffenburger_aris_keller}. 

Exploitation of chemotaxis for biotechnological purposes in mind and envisioning applications in e.g. agriculture (like nitrogen fixation or denitrification), or in mammalian intestinal microbial ecology, in \cite{lauffenburger_rivero_kelly_ford_dirienzo} the authors compare species that undergo growth with different rates and diffusive versus diffusive plus chemotactic motion. 
They conclude: \textit{``Thus, chemotaxis can, when the response is sufficiently strong, overcome both disadvantages of inferior growth kinetics and random motility. [...] At any rate, these results suggest that chemotactic responses might provide a useful means for controlling population dynamics in nonmixed systems. In particular, they provide a way to permit slowly growing populations to coexist with or outcompete faster growing species. Such a situation might be highly desirable in many environmental or biotechnological applications.''} 
For the effectiveness of nutrient-taxis as advantageous dispersal strategy for populations in heterogeneous environments see e.g. \cite{lou_tao_win,cantrell_cosner_lou_advectionmediated}.
% 
% has been investigated and proven in \cite{}
% (see \cite{} and references therein)

Chemotaxis terms in combination with several populations also appear in the context of the host-parasite interaction modelled and analyzed in \cite{pearce_etal} and \cite{tang_tao}, the food-chain model of \cite{liu_ou}, 
or the system in \cite{bendahmane_langlais} that deals with the spread of an epidemic disease. 

For one single species, chemotaxis is described by the 
celebrated model 
by Keller and Segel 
\begin{align}\label{KS}
u_t&= \Delta u - \nabla\cdot (u\nabla v),& \qquad 
\tau v_t&= \Delta v - v +u,%\nn
\end{align}
which has been treated intensively over the last decades. We refer to the surveys \cite{hillen_painter_09,horstmann_03,BBTW}.

Incorporating growth terms into this model, that is, adding $\kappa u-\mu u^2$ to the first equation in \eqref{KS} (studied in \cite{tello_winkler_logsource,winkler_08,winkler_parabolic_logsource,lankeit_evsmoothness}), gives rise to colorful dynamics as witnessed in e.g. \cite{painter_hillen_physicaD}, emphasized by attractor results in \cite{kuto_osaki_sakurai_tsujikawa_12,nakaguchi_efendiev_08,aida_tsujikawa_efendiev_yagi_mimura_06,nakaguchi_osaki_13} or illustrated by recent results on transient growth phenomena in \cite{winkler_ctexceed}, \cite{lankeit_exceed}.

One of the most straightforward generalizations of this model to the situation of several species is to consider two species (or two subpopulations of one species) that react to the same signalling substance they both produce, as occurring in the differentiation of cell-types during slime mold formation in populations of \textit{Dictyostelium discoideum}, cf. e.g. \cite{matsukuma_durston,vasiev_weijer}.

The pure two-species chemotaxis model without growth or competition effects has been introduced in \cite{wolansky} and in particular blow-up of solutions in finite time, 
 known to occur in the single-species situation,  
 has been investigated also for the several species model (see e.g. \cite{horstmann_generalizingKS}), both as to the question of occurrence of blow-up versus global existence (see \cite{biler_espejo_guerra,biler_guerra,espejo_vilches_conca,kurokiba,dickstein,lili})
 %\cite{kurokiba_nagai_ogawa, conca_espejo}
 and of qualitative features of the former, for example whether it occurs simultaneously (\cite{espejo_stevens_velazquez_simultaneous,espejo_stevens_suzuki}) in both species, or nonsimultaneously (\cite{espejo_stevens_velazquez_nonsimultaneous}); for numerical observations pertaining to these models also confer \cite{kurganov_lukacovamedvidova}. 
Furthermore, the long term behaviour of globally existent solutions emanating from small initial data (\cite{zhang_li,yan}) or for chemosensitivity functions capturing saturation effects (\cite{negreanu_tello_ppe}) has been investigated.

 A combination of the two mechanisms introduced so far (Lotka-Volterra type competition and Keller-Segel type chemotaxis) is examined in the two-species chemotaxis-competition model 
 
\begin{equation}\label{cp}
  \begin{cases}
    % about u
    u_t=d_1\Delta u - \chi_1\nabla \cdot (u\nabla w)
    +\mu_1 u(1-u-a_1v), 
    & x\in\Omega,\ t>0, 
\\[1mm]
    % about v
    v_t=d_2\Delta v - \chi_2\nabla \cdot (v\nabla w)
    +\mu_2 v(1-a_2u-v), 
    & x\in\Omega,\ t>0,  
\\[1mm]	
    % about w
    \tau w_t=d_3\Delta w + \alpha u + \beta v - \gamma w, 
    & x\in\Omega,\ t>0, 
\\[1mm]
    % boundary condition
    \nabla u\cdot \nu=\nabla v\cdot \nu = \nabla w\cdot \nu = 0, 
    & x\in\pa \Omega,\ t>0,
\\[1mm]
    % initial condition
    u(x,0)=u_0(x),\; v(x,0)=v_0(x),\; w(x,0)=w_0(x), 
    & x\in\Omega,
  \end{cases} 
\end{equation}
where $u$ and $v$ denote the population densities of two species undergoing chemotaxis in reaction to the signal having concentration $w$, posed in a bounded smooth domain $\Omega\subset\Rone ^n$. Herein, the diffusion rates of species and signal are given by $d_1, d_2, d_3>0$, whereas $\mu_i>0, \chi_i\geq 0, a_i\geq 0$ ($i\in\{1,2\}$) are used to denote the strengths of chemotaxis, growth kinetics and competition for each species, whereas the size of $\alpha>0$, $\beta>0$ and $\gamma>0$ regulates the production of the signal by the first and second species and its decay, respectively.  
This model has first been considered in \cite{tello_winkler} for $\tau =0$, where it was shown for the weakly competitive case, i.e. 
 \begin{equation}\label{cond:ai}
  a_1, a_2\in[0,1),
 \end{equation}
that solutions to \eqref{cp} exist and converge to the coexistence steady state 
\[(u_*,v_*,w_*)=
\left(\frac{1-a_1}{1-a_1a_2},\frac{1-a_2}{1-a_1a_2},\frac{2-a_1-a_2}{1-a_1a_2}\right),\]
provided that 
\begin{equation}\label{strangecondition}
 2(\chi _1+\chi _2)+a_1\mu _2 <\mu _1 \quad \text{and}\quad 2(\chi _1+\chi _2)+a_2\mu _1<\mu _2.
\end{equation}
Result and proof of \cite{tello_winkler} have successfully been extended to a setting of even more species in \cite{wang_li}, the condition being analogous to \eqref{strangecondition}.

This condition seems quite unnatural, because it is not automatically satisfied in the absence of chemotaxis ($\chi _1=\chi _2=0$) and it is the goal of the present paper to replace this condition by a smallness condition on $\frac{\chi _i}{\mu _i}$, $i\in\{1,2\}$ alone and to thus remove the additional condition on $a_1$, $a_2$, $\mu _1$, $\mu _2$ implicitly posed by \eqref{strangecondition}:

\begin{thm}\label{mainth}
Let $n\geq 1$ and $\Omega\subset \Rone ^n$ be a bounded domain with smooth boundary. Let $\tau=0$ and $d_1,d_2,d_3,\alpha, \beta,\gamma \in(0,\infty)$. %Note: \gamma=0 not permissible. 
Let $a_1$, $a_2$ fulfil \eqref{cond:ai} and let $\chi _1\ge 0$, $\chi _2\ge 0$ and $\mu_1,\mu_2>0$ be such that $q_1:=\frac{\chi _1}{\mu _1}$, $q_2:=\frac{\chi _2}{\mu _2}$ 
satisfy the conditions:
\begin{align}\label{conditionqi}
  &q_1\in 
  \big[0,\tfrac{d_3}{2\alpha}\big)\cap
  \big[0,\tfrac{a_1d_3}{\beta}\big), 
\quad
  q_2\in 
  \big[0,\tfrac{d_3}{2\beta}\big)\cap
  \big[0,\tfrac{a_2d_3}{\alpha}\big),
%  &q_1\in 
%  \left(0,\min\left\{\tfrac{1}{2},a_1\right\}\right), \quad
%  q_2\in (0,\min\{\tfrac{1}{2},a_2\}),
\\\label{conditionqi2}
  &a_1a_2d_3^2<(d_3-2\alpha q_1)(d_3-2\beta q_2).
\end{align}
Then the following holds: 
\begin{enumerate}
\item[{\rm (i)}] 
For all nonnegative functions $u_0$, $v_0
\in C(\ol{\Omega})$ satisfying 
$u_0\not\equiv 0$, $v_0\not\equiv 0$, 
there exists a unique 
global-in-time classical solution $(u,v,w)\in\left(C^0(\ol{\Omega}\times[0,\infty))\cap C^{2,1}(\ol{\Omega}\times(0,\infty))\right)^3$ 
of \eqref{cp} such that 
$u>0$, $v>0$ and $w>0$ in $\ol{\Omega}\times(0,\infty)$. 
\item[{\rm (ii)}]
The unique global solution $(u,v,w)$ of \eqref{cp} 
has the following asymptotic behaviour\/{\rm :}
\begin{align*}
  u(t)\to u^{\ast}, 
\quad 
  v(t)\to v^{\ast}, 
\quad
  w(t)\to \frac{\alpha u^\ast+\beta v^\ast}{\gamma}
\quad
  \mbox{as}\ t\to\infty, 
\end{align*}
uniformly in $\Omega$, where 
\begin{align*}
u^{\ast}&=\frac{1-a_1}{1-a_1a_2},\quad \text{and}\quad v^{\ast}=\frac{1-a_2}{1-a_1a_2}.
\end{align*}
\end{enumerate}
\end{thm}

We want to emphasize that \eqref{conditionqi} can indeed be viewed as 
a smallness condition on the relative chemotactic 
strength only and, 
in particular, 
satisfied in 
the chemotaxis-free case ($\chi _1=\chi _2=0$). \\

The case of partially strong competition ($a_1>1>a_2>0$) was considered in \cite{stinner_tello_winkler}. It was shown that solutions exist globally and satisfy $(u(t),v(t))\to  (0,1)$ if $q_1\leq a_1$, $q_2<\frac12$, $\alpha q_1+\max\{q_2,\frac{a_2-a_2q_2}{1-2q_2},\frac{\alpha q_2-a_2q_2}{1-2q_2}\}<1$.
%, whenever $\tau=0$, $d_3=1=\beta$, $d_1,d_2,\alpha,\gamma>0$.

In the fully parabolic system (\eqref{cp} with $\tau =1$), it was proven in 
\cite{bai_winkler} that for sufficiently large values of $\mu _1$, $\mu _2$ global classical solutions converge to the unique positive homogeneous equilibrium  exponentially for $a_1>1>a_2>0$ (and with an algebraic rate if $a_1=1$ and $\mu _2$ is large), and moreover that there are global bounded classical solutions for $n\leq2$ even if the parameters of the system are merely positive.

Furthermore, in spatially one-dimensional domains more insight into qualitative behaviour of the system has been obtained in \cite{wang_yang_zhang} and \cite{wang_zhang_yang_hu}, where global existence of solutions was shown as well as existence of nonconstant steady states by bifurcation analysis. The stability of the bifurcating solutions is investigated there, too, and a time-periodic solution has been found. 
Additionally, the findings have also been illustrated by numerical experiments (\cite[Sec. 4]{wang_zhang_yang_hu}).

In \cite{zhang_li_globalboundedness_twospecies}, for different sensitivity functions satisfying $\chi _i(w)\leq\frac{K_i}{(1+\alpha _i w)^{k_i}}$ for some $k_i>1$, $i\in\{1,2\}$, 
%$h=u+v-w$, . 
global bounded classical solutions were proven to exist under the condition of $\chi , \mu , K$ being sufficiently small (where the meaning of 'sufficiently small' depends on the initial mass).  

In the competition-free case ($a_1=a_2=0$) for sensitivity functions generalizing $\frac{\chi _{i,0}}{(1+w)^k}$ ($k>1$), global existence and boundedness of solutions were obtained, together with a result on asymptotic stability of steady states. (\cite{masaaki_tomomi}) 

A system where the chemoattractant is not a signal substance produced by the population itself
but a nutrient which is consumed, (but which is otherwise similar), has been treated in \cite{zhangzhenbu,wang_wu}.

% \cite{zhangzhenbu}: gen of $h=-u-v$, kinetics: $u(w-u-v)$. $n=1$, GE, steady states. (related model (?) without chemotaxis: \cite{zhangzhenbu_withoutct})\\
% \cite{wang_wu}: $n=1$, $h=-f(w)(u+v)$, $\chi _i=\chi _i・上・(u)$, kinetics: $+u(kf(w)-・趣ｽｸ)$. $d_1, \chi _1$ fixed. regions for $(d_2,・趣ｽｻ_2)$ such that no pos. s.s/ steady states unstable/loc as. stable. \\

Let us finally mention that also (parabolic-elliptic) Keller-Segel type systems of two species and two chemicals have been studied where the signal for each species is produced by the other (\cite{taowin_twospecies_twochemicals}).
 
A parabolic-parabolic-ODE model with connections to chemotaxis-haptotaxis models (see e.g. \cite{chaplain_anderson}) has been treated in \cite{negreanu_tello_ppo}. %$a_1=a_2=d_3=0$, $d_1=d_2=\tau =1$, some $\chi _i(w)$, GE, bdness, convergence (Moreover: connections to chemotaxis-haptotaxis models) 
 
Cross-diffusive effects like those added to the present system by means of the chemotaxis term pose a serious threat to any monotonicity properties one would like to employ  and usually should render comparison arguments useless. 
 Nevertheless, in some situations closely related, comparison methods were employable, for example the proof of the result in \cite{tello_winkler} relies on comparison with solutions to a system of four coupled ODEs. (In fact, a system closely related to \eqref{cp} was considered as example in \cite{negreanu_tello_comparison}, where comparison with solutions of ODE systems has been developed more systematically. However, there the essential coupling arising from the appearance of $u$ and $v$ as source terms for the third equation was not included.)
 
Most successfully, comparison arguments were utilized in deriving the global asymptotic stability of 
in \eqref{cp} for the case of strong competition in \cite{stinner_tello_winkler}.  

The situation considered there is, in a certain sense, easier than the present case, since the limit of one component being $0$ makes lower estimates for this component unnecessary and simplifies the system of inequalities that has to be dealt with during the proof. 

Nevertheless this method of 'eventual comparison', as we would like to call it, 
turns out to be a powerful tool also in the present context and we will use it to derive the desired result. 

More precisely, we shall proceed as follows: At first we will establish global existence and boundedness of the solutions to \eqref{cp}. 
 Then, knowing that limes inferior and limes superior of each component exist and are not infinite, we find differential inequalities that are eventually (that is, on time intervals of the form $(T,\infty )$ for some $T>0$) satisfied and that are accessible to comparison arguments, the corresponding ODE solution converging (almost) regardless of its initial value. This will give us enough information to deduce the precise value of limites superior and inferior and to prove convergence of the solution to the coexistence steady state.

%
%
%
%%==============================================================%% <--- this is exaggerated. Why don't you use a good editor instead? 
%%==============                                  ==============%% 
%%======                      Section2                    ======%%
%%====                                                      ====%%
%%==                                                          ==%%
%%====                    Global existence                  ====%%
%%======                                                  ======%%
%%==============                                  ==============%%
%%==============================================================%%
%
%
\section{Global existence}
%
%------------------------------------------------------
%
%		local existence
%
%------------------------------------------------------
\begin{lem}\label{lemlocal}
Let $n\geq 1$ and $\Omega\subset\Rone^n$ be a bounded domain with smooth boundary, let $\chi_1,\chi_2,a_1,a_2\geq 0$, $d_1, d_2, d_3, \alpha, \beta, \gamma, \mu_1, \mu_2\in(0,\infty)$.
Suppose that $u_0$, $v_0\in C(\ol{\Omega})$ are 
nonnegative such that 
$u_0\not\equiv 0$, $v_0\not\equiv 0$. 
Then there exist $\tmax\in(0,\infty]$ and 
a unique classical solution $(u,v,w)$ of \eqref{cp} on $\Omega\times[0,\tmax)$ which belongs to 
$\left(C(\ol{\Omega}\times[0,\tmax))\cap 
C^{2,1}(\ol{\Omega}\times(0,\tmax))\right)^3$, such that moreover the following extensibility criterion holds:
\begin{align*}%\label{tmax}
  \text{Either\ }\tmax=\infty 
\quad 
  \text{or }
\quad 
  \limsup_{t\nearrow \tmax}
  \big(\lp{\infty}{u(\cdot,t)}+\lp{\infty}{v(\cdot,t)}\big)=\infty.
\end{align*}
Furthermore, $u$, $v$ and $w$ are positive in $\ol{\Omega}\times(0,\tmax)$.
\end{lem}
\begin{proof}
The proof is the same as \cite[Lemma 2.1]{stinner_tello_winkler}. 
%Although uniqueness might be obtained from the Banach-fixed point reasoning rather than from the differential inequality 'derived' there.
\end{proof}
In the following lemma we infer boundedness of the solutions by a comparison argument and hence, in accordance with the above extensibility criterion, global existence. 
In its proof and also later on, given $d_1>0$, $d_2>0$ and $w$, we denote by $\Lone $ and $\Ltwo $ the operators defined by
\begin{align}
  \Lone \utilde:=d_1\Delta \utilde-\chi_1\nabla \utilde\cdot \nabla w,
\qquad
  \Ltwo \vtilde:=d_2\Delta \vtilde-\chi_2\nabla \vtilde\cdot \nabla w. \label{defdiffop}
\end{align}
%
%--------------------------------------------------------
%
%		global existence
%
%--------------------------------------------------------
%
\begin{lem}\label{lemglobal}
Suppose that the assumptions of Theorem \ref{mainth} are satisfied and that $\tmax$, $u$, $v$, $w$ are as given by Lemma \ref{lemlocal}.
Then $\tmax=\infty$ and both $u$ and $v$ are bounded in 
$\Omega\times(0,\infty)$.
\end{lem}
%
%-----------------------------------------proof
\begin{proof}
Making use of \eqref{defdiffop}, from the first and third equation of \eqref{cp} 
we obtain 
\begin{align*}%\label{maineq}
  u_t&=\Lone u - \chi_1 u \frac{\gamma w-\alpha u-\beta v}{d_3}
  + \mu_1 u(1-u-a_1v)
\\\notag
     &=\Lone u 
     + \mu_1 u\left(1-\Big(1-\frac{\alpha q_1}{d_3}\Big)u
     -\Big(a_1-\frac{\beta q_1}{d_3}\Big)v
     -\frac{\gamma q_1}{d_3} w\right),
\end{align*}
wherein the choice of $q_1$ then implies
\begin{align*}
u_t     &\leq \Lone u
     + \mu_1u
     \left(1-\Big(1-\frac{\alpha q_1}{d_3}\Big)u\right).
\end{align*}
We choose $\ol{u}\in(0,\infty)$ 
such that $\lp{\infty}{u_0}\leq \ol{u}$, 
and denote by $y_1:[0,\infty)\to\Rone$ 
the function solving 
\begin{align*}
\begin{cases}
  y_1'=\mu_1y_1\left(1-
  \big(1-\frac{\alpha q_1}{d_3}\big)y_1\right),
\\
  y_1(0)=\ol{u} 
\end{cases}
\end{align*}
which satisfies
\begin{align*}
  y_1(t)\to \frac{d_3}{d_3-\alpha q_1} 
\quad 
  \mbox{as}\ t\to\infty.
\end{align*}
By a comparison theorem we obtain 
\begin{align*}
  \limsup_{t\nearrow \tmax}u(t)\leq 
  \limsup_{t\nearrow \tmax}y_1(t)=\frac{d_3}{d_3-\alpha q_1}
\end{align*}
Treating the second equation of \eqref{cp} 
in a similar fashion we get
\begin{align*}
  v_t\leq \Ltwo v 
  +\mu_2v\left(1
  -\Big(1-\frac{\beta q_2}{d_3}\Big)v\right),
\end{align*}
and analogously conclude 
\begin{align*}
  \limsup_{t\nearrow \tmax} v(t)\leq \frac{d_3}{d_3-\beta q_2}.
\end{align*}
By the extensibility criterion we obtain $\tmax=\infty$. 
\end{proof}
%-------------------------------------------end
%
%%==============================================================%%
%%==============                                  ==============%%
%%======                      Section3                    ======%%
%%====                                                      ====%%
%%==                                                          ==%%
%%====                 Asymptotic stability                 ====%%
%%======                                                  ======%%
%%==============                                  ==============%%
%%==============================================================%%
%
%
\section{Global asymptotic stability}
%
%
%
%------------------------------------------------------
%
%		Definition of L_1 etc.
%
%------------------------------------------------------
%
\noindent 
Since Lemma \ref{lemglobal} guarantees that $u$ and $v$ exist globally and are bounded and nonnegative, it is possible to 
define nonnegative finite real numbers $L_1$, $l_1$, $L_2$, $l_2$ by
\begin{align}
  L_1:=&\limsup_{t\to \infty}
       \left(\max_{x\in\ol{\Omega}}u(x,t)\right),
\qquad
  l_1:=\liminf_{t\to \infty}
       \left(\min_{x\in\ol{\Omega}}u(x,t)\right),
\nn
\\
  L_2:=&\limsup_{t\to \infty}
       \left(\max_{x\in\ol{\Omega}}v(x,t)\right),
\qquad
  l_2:=\liminf_{t\to \infty}
       \left(\min_{x\in\ol{\Omega}}v(x,t)\right).
\label{defLl}
\end{align}
From the definition we have that 
for all $\ep>0$ there exists $T_\ep>0$ such that
\begin{align}\label{estiu}
  l_1-\ep<u(x,t)<L_1+\ep,
\qquad
  l_2-\ep<v(x,t)<L_2+\ep
\end{align}
hold for all $t>T_\ep$ and all $x\in \Omega$. 
By the maximum principle applied to %Reference: Pucci/Serrin: The maximum Principle, Sec. 2.8 (but probably not necessary to mention, since well-known).
\begin{align*}
\begin{cases}
  -d_3\Delta w + \gamma w = \alpha u + \beta v & \mbox{in} \ \Omega,
\\
  \nabla w\cdot \nu=0 & \mbox{on} \ \pa\Omega, 
\end{cases}
\end{align*}
we have
\begin{align*}
  \min_{\xi\in\ol{\Omega}}(\alpha u(\xi,t)+\beta v(\xi,t))
  \leq
  \gamma  w(x,t)
  \leq
  \max_{\xi\in\ol{\Omega}}(\alpha u(\xi,t)+\beta v(\xi,t)) 
\quad
  \text{for all }\ t>0,\ x\in\Omega .
\end{align*}
Consequently, we obtain from \eqref{estiu} that
for all $\ep>0$ there exists $T_\ep>0$ such that 
\begin{align}\label{estiw}
  \alpha l_1+\beta l_2-2\ep < \gamma  w(x,t)<\alpha L_1+\beta L_2+2\ep\qquad \text{ for all } t>T_\ep \text{ and for all } x\in\Omega. 
\end{align}
%
%

%
%-------------------------------------------------------
%
%		estimate for L_1 and l_1.
%
%-------------------------------------------------------
%
Employing comparison arguments on ultimate time intervals, we derive first estimates for the quantities defined in \eqref{defLl}. 
\begin{lem}\label{lemestiL1}
Under the assumptions of Theorem \ref{mainth} and with \eqref{defLl}, the following inequalities hold:
\begin{align}\label{estiL1l1}
  L_1\leq 
  \dfrac{\left(
  d_3-\alpha q_1 l_1-a_1d_3l_2
  \right)_+}{d_3-\alpha q_1}
\quad 
  {\rm and}
\quad
  l_1\geq \dfrac{d_3-\alpha q_1 L_1-a_1d_3L_2}{d_3-\alpha q_1}.
\end{align}
\end{lem}
%
%-------------------------------------proof
\begin{proof}
Recalling that from the first and third equation 
of \eqref{cp} and using the same notation as in \eqref{defdiffop} we have
\begin{align}\label{eq:ueq_tobereplacedforv}
u_t-\Lone u&= \mu_1 u\left(1-\Big(1-\frac{\alpha q_1}{d_3}\Big)u
     -\Big(a_1-\frac{\beta q_1}{d_3}\Big)v
     -\frac{\gamma q_1}{d_3} w\right),
\end{align}
we let $\ep>0$ and make use of \eqref{estiu} and \eqref{estiw} 
to find $T_\ep>0$ such that 
\begin{align*}
  u_t-\Lone u&\leq\mu_1 u\left(1-\Big(1-\frac{\alpha q_1}{d_3}\Big)u
     -\Big(a_1-\frac{\beta q_1}{d_3}\Big)(l_2-\ep)
     -\frac{q_1}{d_3} (\alpha l_1+\beta l_2-2\ep)\right) 
\\
  &= \mu_1u\left(1-\Big(1-\frac{\alpha q_1}{d_3}\Big)u-\frac{\alpha q_1}{d_3}l_1
  -a_1l_2
  +\Big(a_1+(2-\beta)\frac{q_1}{d_3}\Big)\ep\right)
\quad 
  \mbox{on}\ (T_\ep,\infty), 
\end{align*}
where for the estimates we relied on nonnegativity of $a_1-\frac{\beta q_1}{d_3}$ as guaranteed by \eqref{conditionqi}.
We choose $\ol{u_\ep}\in(0,\infty)$ such that 
$u(\cdot,T_\ep)\leq\ol{u_\ep}$ in $\Omega$ 
and we denote by $\ol{z}:[T_\ep,\infty)\to\Rone$ 
the function solving 
\begin{align*}
\begin{cases}
  \ol{z}'=\mu_1\ol{z}\left(1-\Big(1-\frac{\alpha q_1}{d_3}\Big)\ol{z}
  -\frac{\alpha q_1}{d_3}l_1-a_1l_2
  +\Big(a_1+(2-\beta)\frac{q_1}{d_3}\Big)\ep\right)\quad \text{in } (T_\ep,\infty),
\\
  \ol{z}\left(T_\ep\right)=\ol{u_\ep}, 
\end{cases}
\end{align*}
which satisfies 
\begin{align*}
  \ol{z}\left(t\right)\to
  \frac{\left(d_3-\alpha q_1l_1-a_1d_3l_2
  +\Big(a_1d_3+(2-\beta)q_1\Big)\ep\right)_+}{d_3-\alpha q_1}
\quad 
  \mbox{as}\ t\to\infty. 
\end{align*}
By comparison we obtain
\begin{align}\label{equL1}
  L_1=\limsup_{t\to\infty}
  \left(\max_{x\in\ol{\Omega}}u(x,t)\right)
  \leq 
  \limsup_{t\to\infty}\ol{z}(t)=
  \dfrac{\left(d_3-\alpha q_1l_1-a_1d_3l_2
  +\Big(a_1d_3+(2-\beta)q_1\Big)\ep\right)_+}{d_3-\alpha q_1}.
\end{align}
%
%
%\tc{
%We obtain \eqref{equL1} for all $\ep$, which means 
%}
%\begin{align*}
%  L_1\leq \dfrac{\left(1-a_1l_2-q_1l_1\right)_+}{1-q_1}.
%\end{align*}
%
%
On the other hand, making use of 
the other estimates in \eqref{estiu} and \eqref{estiw} and again of \eqref{conditionqi},
we have 
\begin{align*}
  u_t-\Lone u&= \mu_1 u\left(1-\Big(1-\frac{\alpha q_1}{d_3}\Big)u
     -\Big(a_1-\frac{\beta q_1}{d_3}\Big)v
     -\frac{\gamma q_1}{d_3} w\right)
\\
  &\geq \mu_1u\left(1-\Big(1-\frac{\alpha q_1}{d_3}\Big)u-\Big(a_1-\frac{\beta q_1}{d_3}\Big)\left(L_2+\ep\right)
        -\frac{q_1}{d_3}\left(\alpha L_1+\beta L_2+2\ep\right)\right)
\\
  &= \mu_1u\left(1-\Big(1-\frac{\alpha q_1}{d_3}\Big)u-\frac{\alpha q_1}{d_3}L_1-a_1L_2-\Big(a_1-(2-\beta)\frac{q_1}{d_3}\Big)\ep\right)\quad 
  \mbox{on}\ (T_\ep,\infty).
\end{align*}
Choosing $\underline{u_\ep}>0$ such that 
$u\left(\cdot,T_\ep\right)\geq\underline{u_\ep}$ in $\Omega$ and 
denoting by $\ul{z}:[T_\ep,\infty)\to\Rone$ 
the function solving 
\begin{align*}
\begin{cases}
  \ul{z}'=\mu_1\ul{z}
  \left(1-\Big(1-\frac{\alpha q_1}{d_3}\Big)\ul{z}-\frac{\alpha q_1}{d_3}L_1-a_1L_2-\Big(a_1-(2-\beta)\frac{q_1}{d_3}\Big)\ep\right)\quad \text{in } (T_\ep,\infty),
\\
  \ul{z}\left(T_\ep\right)=\underline{u_\ep},
\end{cases}
\end{align*}
which satisfies 
\begin{align*}
  \ul{z}\left(t\right)\to 
  \frac{\left(d_3-\alpha q_1L_1-a_1d_3L_2
  -\Big(a_1d_3+(2-\beta)q_1\Big)\ep\right)_+}{d_3-\alpha q_1}
\quad
  \mbox{as}\ t\to\infty.
\end{align*}
we obtain from the comparison theorem that 
\begin{align}\label{equl1}
  l_1=
  \liminf_{t\to\infty}
  \left(\min_{x\in\ol{\Omega}}u\left(x,t\right)\right)
  \geq \liminf_{t\to\infty}\ul{z}\left(t\right)
  \geq \frac{d_3-\alpha q_1L_1-a_1d_3L_2-(a_1d_3+(2-\beta)q_1)\ep}{d_3-\alpha q_1}
  %-\frac{a_1d_3+(2-\beta)q_1}{d_3-\alpha q_1}\ep}
\end{align}
holds. 
Because $\ep>0$ was arbitrary, 
\eqref{equL1} and \eqref{equl1} entail \eqref{estiL1l1}. 
\end{proof}
%
%----------------------------------------------end
%
%----------------------------------------------------------
%
%		estimate for L_2 and l_2
%
%----------------------------------------------------------
%
\begin{lem}\label{lemestiL2}
Under the assumptions of Theorem \ref{mainth} and with notation as in \eqref{defLl}, the following inequalities hold:
\begin{align*}
  L_2\leq \dfrac{\left(
  d_3-a_2d_3l_1-\beta q_2 l_2
  \right)_+}{d_3-\beta q_2}
\quad 
  {\rm and}
\quad
  l_2\geq \frac{d_3-a_2d_3L_1-\beta q_2L_2}{d_3-\beta q_2}.
\end{align*}
\end{lem}
%
%---------------------------------------------proof
\begin{proof}
Repeating the arguments from the proof of Lemma \ref{lemestiL1}, this time with
\begin{align*}
v_t-\Ltwo v&= \mu_2 v\left(1-\Big(1-\frac{\beta q_2}{d_3}\Big)v
     -\Big(a_2-\frac{\alpha q_2}{d_3}\Big)u
     -\frac{\gamma q_2}{d_3} w\right),
\end{align*}
instead of \eqref{eq:ueq_tobereplacedforv},
leads to  Lemma \ref{lemestiL2}.
\end{proof}
%----------------------------------------------end
%
Before we continue, let us briefly verify that the differences appearing in the numerators of the upper bounds for $L_1$ and $L_2$ are already nonnegative, so that we can neglect the positive part operator $(\cdot)_+$. Later on, this will allow us to conclude convergence to the non-zero equilibrium state.
%-------------------------------------------------------
%
%		SEmi-Final lemma
%
%-------------------------------------------------------
%
\begin{lem}\label{lempositivnumerator}
Under the assumptions of Theorem \ref{mainth} and with notation as in \eqref{defLl}, 
\begin{align*}
 d_3-\alpha q_1l_1-a_1d_3l_2\geq 0\quad\text{ and }\quad  d_3-a_2d_3l_1-\beta q_2 l_2\geq0.
\end{align*}
\end{lem}
%
%----------------------------------------proof
\begin{proof}
We work along the lines of a contradiction argument to show that the undesired cases can in fact not appear. If 
\begin{align}\label{part1-1}
  d_3-\alpha q_1l_1-a_1 d_3 l_2<0
\qquad\text{and}\qquad
  d_3-a_2 d_3 l_1-\beta q_2l_2<0, 
\end{align}
we obtain from 
Lemma \ref{lemestiL1}, 
Lemma \ref{lemestiL2} and the nonnegativity 
of $u$ and $v$ that 
\begin{align*}
 0\leq l_1\leq L_1\leq 0 \qquad \text{and}\qquad 0\leq l_2\leq L_2\leq 0.
\end{align*}
Plugging this back into \eqref{part1-1} yields the contradiction $d_3<0$. In the case
\begin{align*}%\label{part1-2}
  d_3-\alpha q_1l_1-a_1 d_3 l_2\geq 0
\qquad\text{and}\qquad
  d_3-a_2 d_3 l_1-\beta q_2l_2<0, 
\end{align*}
Lemma \ref{lemestiL1}, Lemma \ref{lemestiL2} and the nonnegativity of $v$ show that
\begin{align}\label{part2-1}
L_2=l_2=0\quad\text{and}\quad L_1\leq\frac{d_3-\alpha q_1 l_1}{d_3-\alpha q_1},\quad l_1\geq\frac{d_3-\alpha q_1 L_1}{d_3-\alpha q_1}.
\end{align}
Herein, the first two inequalities lead to
\begin{align*}
(d_3-\alpha q_1) L_1&\leq d_3-\alpha q_1 l_1,\\
(d_3-\alpha q_1)l_1&\geq d_3-\alpha q_1 L_1,
\end{align*}
which implies
\begin{align*}
(d_3-2\alpha q_1)(L_1-l_1)\leq 0.
\end{align*}
Due to $q_1<\frac{d_3}{2\alpha}$ (by \eqref{conditionqi}) this yields $L_1=l_1$. Together with \eqref{part2-1} this equality shows 
\begin{align*}
L_1=\frac{d_3-\alpha q_1 L_1}{d_3-\alpha q_1},
\end{align*}
which leads to $l_1=L_1=1$ because of $d_3>0$. Making use of this and \eqref{part2-1}, we conclude that the second inequality from Lemma \ref{lemestiL2} implies the contradiction
\begin{align*}
0=l_2\geq\frac{d_3-a_2d_3L_1-\beta q_2 L_2}{d_3-\beta q_2}=\frac{(1-a_2)d_3}{d_3-\beta q_2}>0,
\end{align*}
since $d_3>0$, $d_3>\beta q_2$ due to \eqref{conditionqi} and, by \eqref{cond:ai}, $1>a_2$. The case 
\begin{align*}
  d_3-\alpha q_1l_1-a_1 d_3 l_2< 0
\qquad\text{and}\qquad
  d_3-a_2 d_3 l_1-\beta q_2l_2\geq0
\end{align*}
can be treated in a similar fashion relying on the facts that $q_2<\frac{d_3}{2\beta}$ by \eqref{conditionqi} and $a_1<1$ by \eqref{cond:ai} to obtain the contradiction $0=l_1>0$.
\end{proof}

Having explicit bounds for $L_1$ and $L_2$ at hand, we can now calculate the exact values of $L_1$ and $L_2$.
%-------------------------------------------------------
%
%		Final lemma
%
%-------------------------------------------------------
%
\begin{lem}\label{lemuast}
Let the assumptions of Theorem \ref{mainth} be satisfied. 
Then 
\begin{align*}
  L_1=l_1=u^{\ast},
\quad
  L_2=l_2=v^{\ast},
\end{align*}
where 
\begin{align*}
  u^{\ast}:=\dfrac{1-a_1}{1-a_1a_2},
\quad
  v^{\ast}:=\dfrac{1-a_2}{1-a_1a_2}.
\end{align*}
Moreover the solution of \eqref{cp} converges to 
nontrivial steady states, i.e., 
\begin{align*}
  u\left(t\right)\to u^{\ast}, 
\quad 
  v\left(t\right)\to v^{\ast}, 
\quad
  w\left(t\right)\to \frac{\alpha u^\ast+\beta v^\ast}{\gamma}
\end{align*}
as $t\to\infty$, uniformly in $\Omega$. 
\end{lem}
%
%----------------------------------------proof
\begin{proof}
At first, we shall prove that in fact the solutions converge, namely that 
\begin{align*}
  L_1=l_1,\qquad L_2=l_2
\end{align*}
hold. 
Thanks to Lemmata \ref{lemestiL1}, \ref{lemestiL2} and \ref{lempositivnumerator}
we know that the inequalities 
\begin{align}
  \left(d_3-\alpha q_1\right)L_1\leq d_3-\alpha q_1l_1-a_1d_3l_2,\label{fir}
\\
  \left(d_3-\beta q_2\right)L_2\leq d_3-a_2d_3l_1-\beta q_2l_2,\label{sec}
\\
  \left(d_3-\alpha q_1\right)l_1\geq d_3-\alpha q_1L_1-a_1d_3L_2,\label{thi}
\\
  \left(d_3-\beta q_2\right)l_2\geq d_3-a_2d_3L_1-\beta q_2L_2\label{fou}
\end{align}
hold. From \eqref{fir} and \eqref{thi} we extract 
\begin{align*}
  \left(d_3-\alpha q_1\right)\left(L_1-l_1\right)&\leq \alpha q_1\left(L_1-l_1\right)+a_1d_3\left(L_2-l_2\right).
\end{align*}
Re-ordering this inequality while paying attention to \eqref{conditionqi}, we see that
\begin{align}\label{L1-l1}
  L_1-l_1\leq \dfrac{a_1d_3}{d_3-2\alpha q_1}\left(L_2-l_2\right).
\end{align}
Taking into account \eqref{sec} and \eqref{fou}, 
from a similar argument we obtain 
\begin{align}\label{L2-l2}
  L_2-l_2\leq \dfrac{a_2d_3}{d_3-2\beta q_2}\left(L_1-l_1\right).
\end{align}
Combination of \eqref{L1-l1} and \eqref{L2-l2} shows
\begin{align*}
  L_1-l_1\leq 
  \dfrac{a_1d_3}{d_3-2\alpha q_1}\cdot\dfrac{a_2d_3}{d_3-2\beta q_2}\left(L_1-l_1\right),
\end{align*}

which, by the smallness condition on $q_1$, $q_2$ in \eqref{conditionqi2}, implies $L_1=l_1$. 
Plugging this result into \eqref{L2-l2} 
also yields $L_2=l_2$. 
Lastly, we shall prove $l_1=u^{\ast}$ 
and $l_2=v^{\ast}$. 
From \eqref{fir}, \eqref{thi}, $L_1=l_1$ and $L_2=l_2$ we have
\begin{align*}
  \left(d_3-\alpha q_1\right)l_1 &= d_3-\alpha q_1l_1-a_1d_3 l_2,
\end{align*}
which, after re-ordering, leads to 
\begin{align}\label{l1}
  l_1=1-a_1l_2.
\end{align}
Similarly we obtain
from \eqref{sec}, \eqref{fou}, $L_1=l_1$ and $L_2=l_2$ that 
\begin{align}\label{l2}
  l_2=1-a_2l_1.
\end{align}
Combination of \eqref{l1} and \eqref{l2} therefore leads to 
\begin{align*}%\label{eql1l2}
  l_1=\dfrac{1-a_1}{1-a_1a_2}=u^{\ast},
\\\notag
  l_2=\dfrac{1-a_2}{1-a_1a_2}=v^{\ast}.
\end{align*}
$L_1=l_1=u^{\ast}$ and $L_2=l_2=v^{\ast}$ imply 
that 
\begin{align*}
  u\left(t\right)\to u^{\ast}, \quad v\left(t\right)\to v^{\ast}\quad 
  \mbox{as}\ t\to\infty,
\end{align*}
uniformly in $\Omega$. 
Finally, accordance with \eqref{estiw}, for any 
$\ep>0$ there is $T_\ep>0$ such that 
\begin{align*}
 \alpha u^\ast+\beta v^\ast-2\ep < \gamma w\left(x,t\right) < \alpha u^\ast+\beta v^\ast+2\ep
\quad
  \mbox{for all }\  (x,t)\in \Omega\times (T_\ep,\infty)
\end{align*}
and hence $w(\cdot,t)\to \frac{\alpha u^\ast+\beta v^\ast}{\gamma}$ as $t\to\infty$, uniformly in $\Omega$. 
\end{proof}
%--------------------------------------------end
%
With this lemma, we actually have completed the proof of Theorem \ref{mainth}.
%---------------------------------------------------------
%
%		proof of main theorem
%
%----------------------------------------------------------
%
\begin{proof}[{\rm \bf Proof of Theorem \ref{mainth}}]
Part (i) follows from Lemma \ref{lemglobal}, while (ii) is contained in Lemma \ref{lemuast}.
\end{proof}

\section{Numerical experiments}
We have numerically implemented system \eqref{cp} for $\tau=0$ in the domain $\Omega=(0,3)\subset \mathbb{R}$, employing a finite difference discretization together with an explicit Euler scheme in time. Mesh size was given by delta\_x=0.005 and delta\_t=0.00001. In this section we want to illustrate some of the solution behaviour that can be observed in the simulation.

In all of the following experiments we have chosen $d_1=d_2=d_3=1$ and $\alpha=\beta=\gamma=1$. 

\textbf{First observation: The expected behaviour.} Starting from initial data 
\[u_0(x)=2+1.5\cos(\pi\cdot(x-0.6))\text{ and }v_0(x)=(1-x)^2,\] and for parameters $\mu_1=\mu_2=2$, $a_1=0.6$, $a_2=0.4$ and $q_1=0.2$, $q_2=0.1$, solutions are revealed to converge to the (constant) coexistence steady state. 

Fig. \ref{fig-expected} shows the graphs of $u(\cdot,t_0)$, $v(\cdot,t_0)$, $w(\cdot,t_0)$ after 
1, 5000, 15000, 100000, 200000, 1000000 time steps (i.e. for $t_0=0.00001$, $t_0=0.05$, $t_0=0.15$, $t_0=1$, $t_0=2$ and $t_0=10$, respectively). 
\begin{figure}[!htbp]
\includegraphics[width=0.33\textwidth]{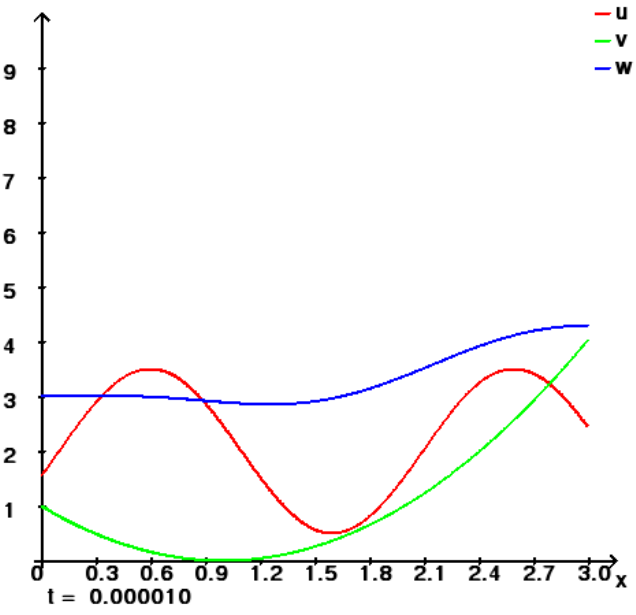}
\includegraphics[width=0.33\textwidth]{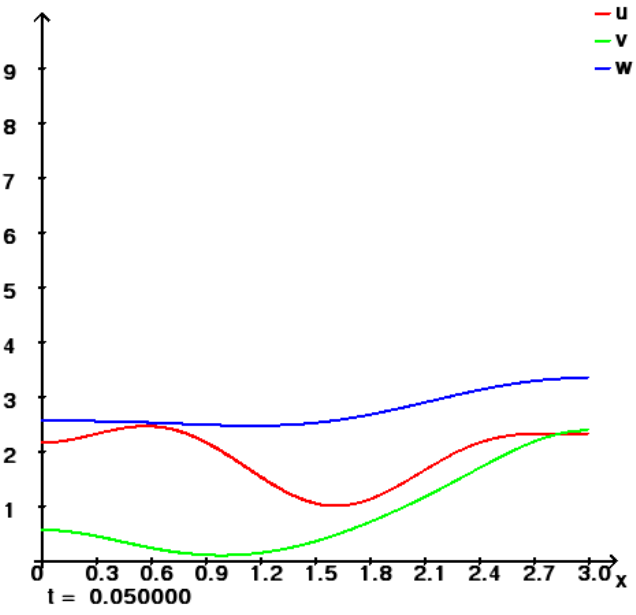}
\includegraphics[width=0.33\textwidth]{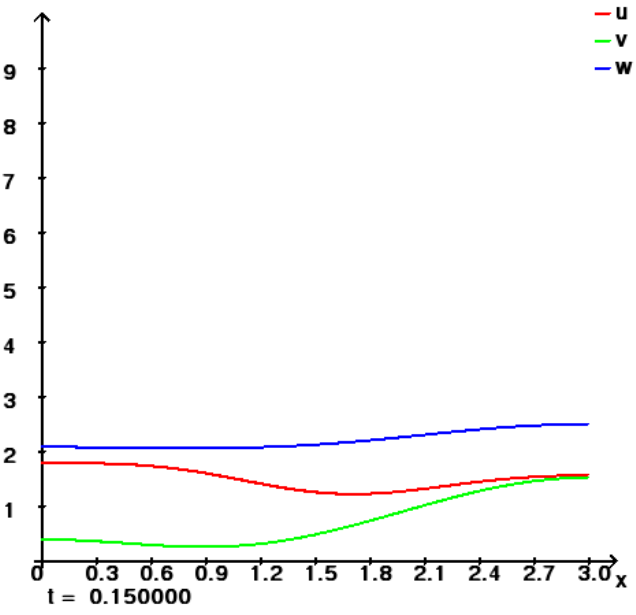}
%\end{figure}
%\begin{figure}[!htbp]
\includegraphics[width=0.33\textwidth]{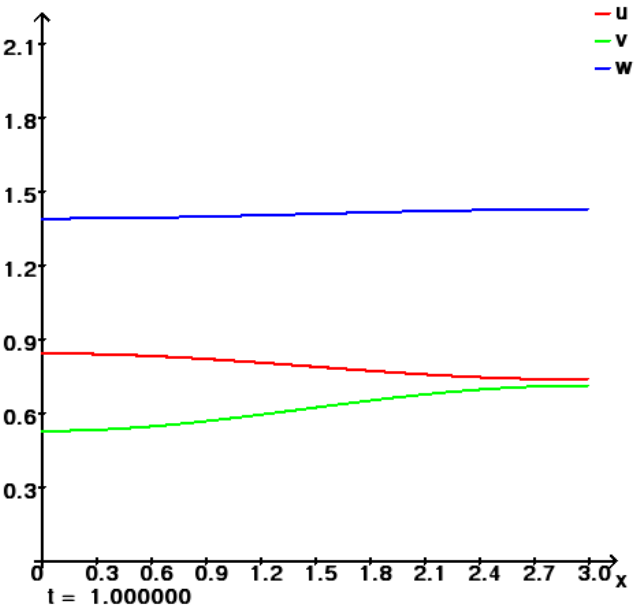}
\includegraphics[width=0.33\textwidth]{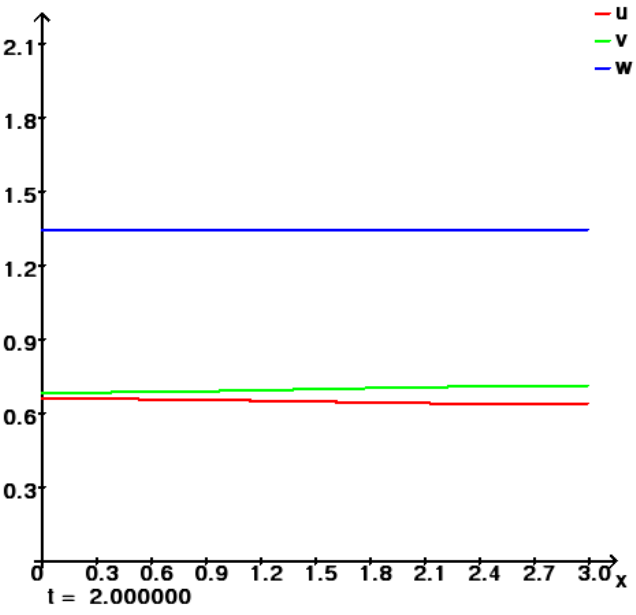}
\includegraphics[width=0.33\textwidth]{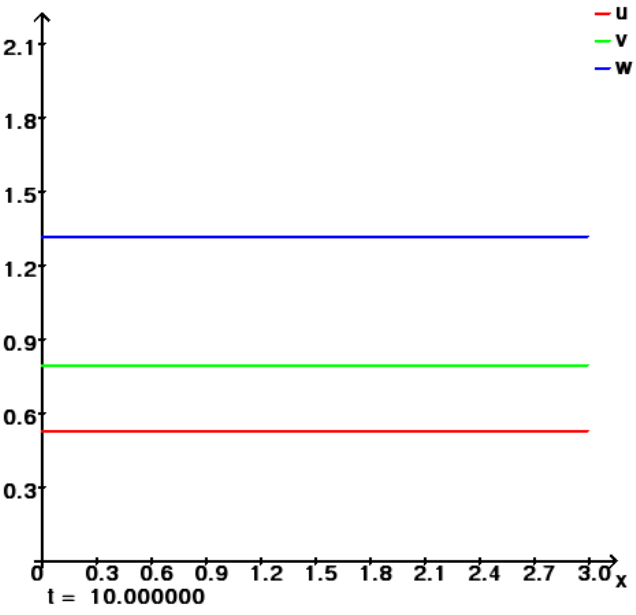}
\caption{The expected solution behaviour}\label{fig-expected}%label has to be after caption, otherwise number referred to will be section number
\end{figure}

Because all parameters lie in the appropriate range, Theorem \ref{mainth} verifies that the convergence observed is the correct long-term behaviour. 

\textbf{Second observation: Large $q_i$.} Choosing $q_1=50$ and $q_2=0.1$, but otherwise the same parameters as in the first scenario, we witness the following ``large-time'' behaviour; pictures are taken at after 60000, 100000, 200000, 700000 time steps.

\begin{figure}[!hbtp]
\includegraphics[width=0.24\textwidth]{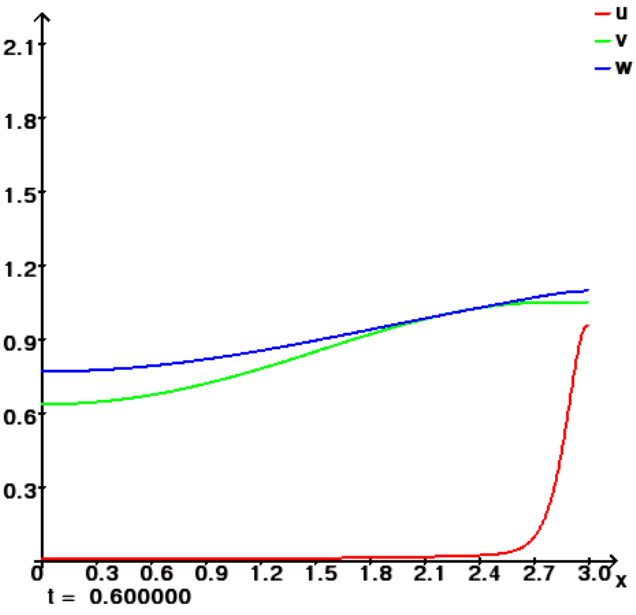}
\includegraphics[width=0.24\textwidth]{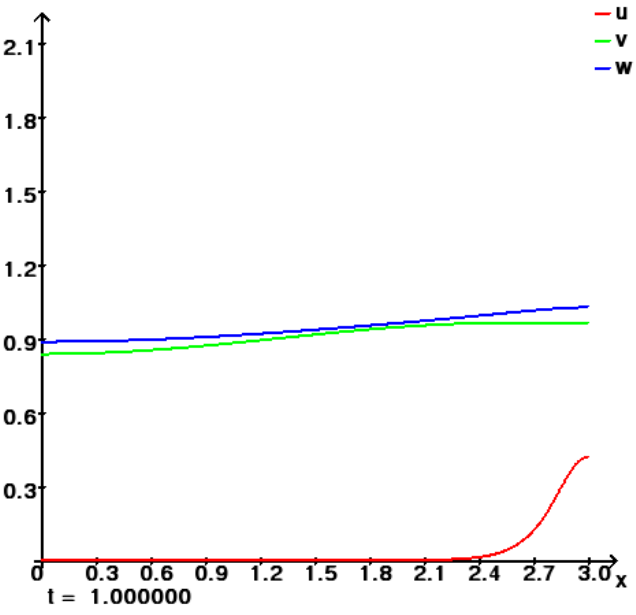}
\includegraphics[width=0.24\textwidth]{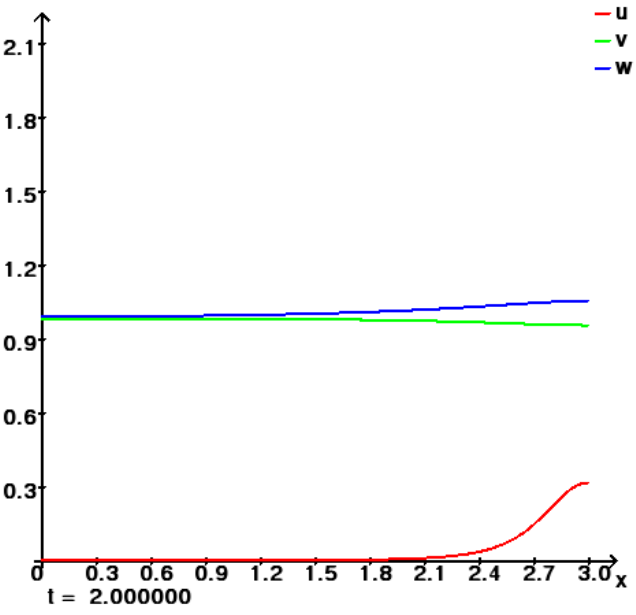}
\includegraphics[width=0.24\textwidth]{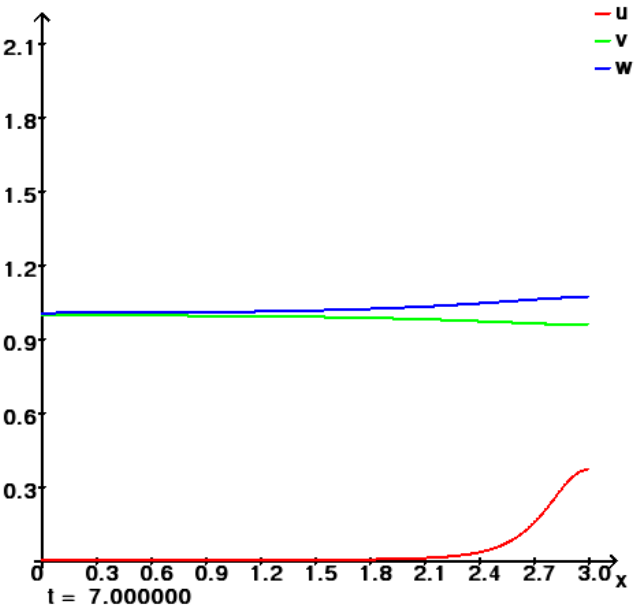}
\caption{Large values of $q_1$.}\label{fig-largeq}
\end{figure}
The solutions still can be perceived to converge; however, the limiting profile is significantly different and not even spatially homogeneous, as can be seen in the last graph of Figure \ref{fig-largeq}, where the solutions already have stabilized (which is indicated by the negligible changes between $t=2$ and $t=7$). 

Whereas on the left part of the domain, the first species seems to have vanished entirely, it has survived on the right, lured there by the comparatively high concentrations of the chemoattractant which go back to the high initial deployment of the second species there.

Altogether, this contrasting appearance of the limit 
demonstrates  
that at least some smallness condition on $q_1$, $q_2$ is essential for the validity of the conclusion of Theorem \ref{mainth}.

\textbf{Third observation: Behaviour on small time-scales.} 
The following pictures are taken from the experiment performed for the second observation, that is for the same choice of parameters as in the preceding subsection. This time, we want to have a closer look at early solution behaviour and correspondingly in Fig. \ref{fig-smalltime} depict the solution at time steps 
500, 600, 676, 1400, 13000 and 30000. 
\begin{figure}[!hbtp]
\includegraphics[width=0.33\textwidth]{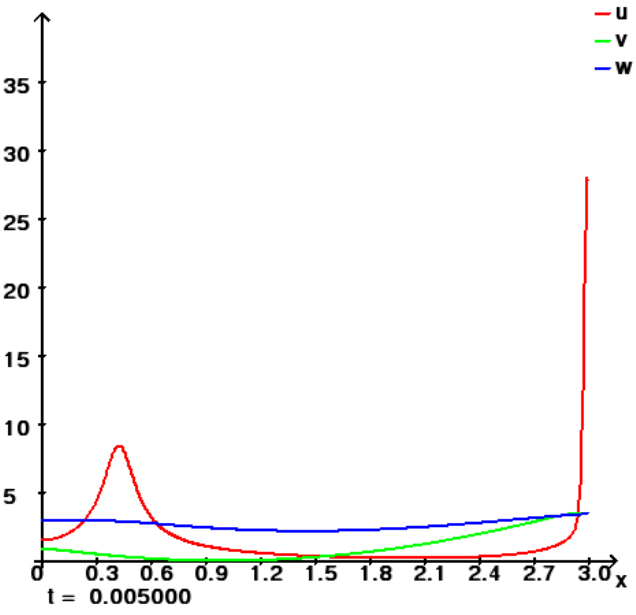}
\includegraphics[width=0.33\textwidth]{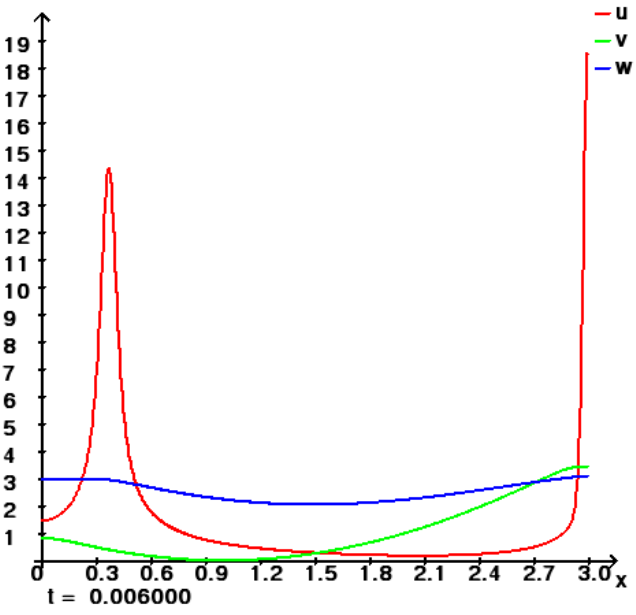}
\includegraphics[width=0.33\textwidth]{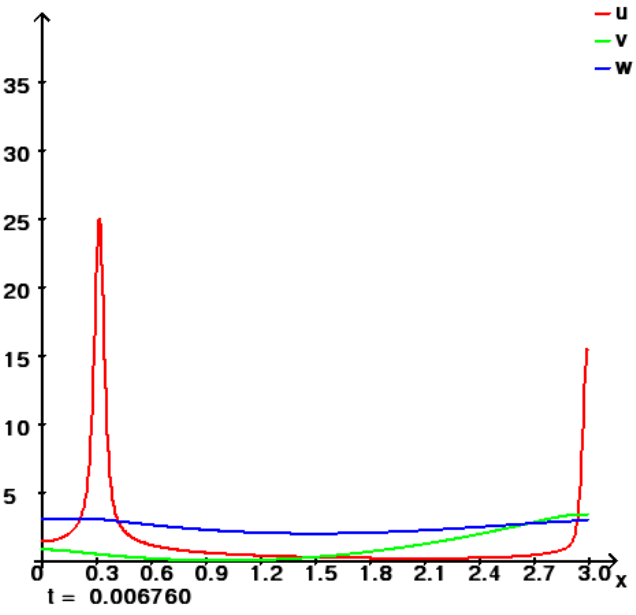}
\includegraphics[width=0.33\textwidth]{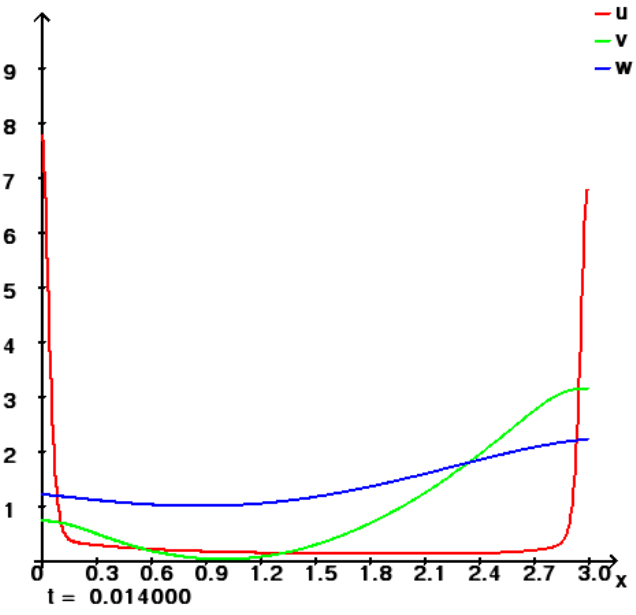}
\includegraphics[width=0.33\textwidth]{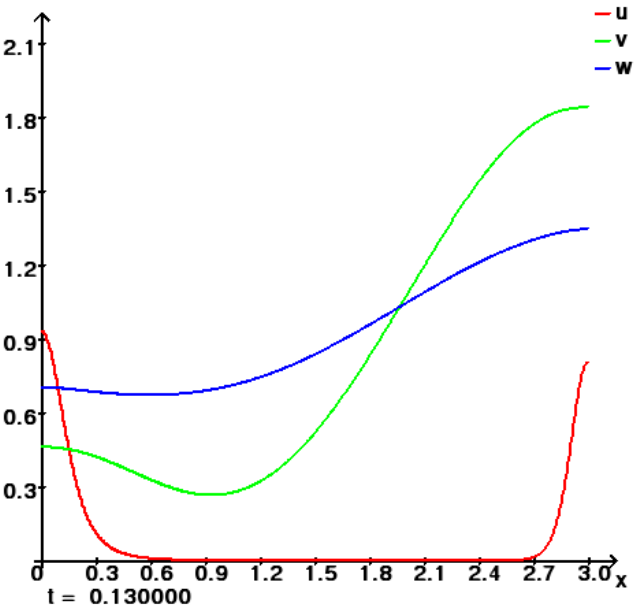}
\includegraphics[width=0.33\textwidth]{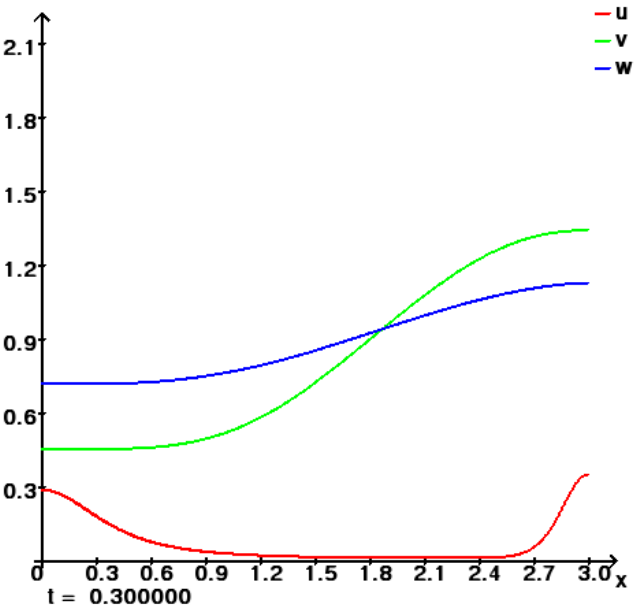}
\caption{Solution at smaller times. {\footnotesize (Note the changes of scale on the vertical axis.)}}\label{fig-smalltime}
\end{figure}

On these small time scale, 
that should be more representative 
for the actual dynamical evolution of species 
obeying model \eqref{cp} that the large-time 
limiting behaviour of solutions, 
diverse phenomena occur. 
In Figure \ref{fig-smalltime}, 
the simultaneous emergence
and movement of large aggregates of 
the chemotactically strongly active population can be 
observed as response to the emission of 
the joint signal by both species 
while they compete and the second population 
more slowly reacts to 
the same concentration gradients. 
It is worth noting that in this process 
astonishingly high concentrations can be 
reached; cf. Figure \ref{fig-max}, 
where the spatial maxima of $u$ and $v$ are 
plottet versus time $t\in[0,0.02\ldots]$, 
and although the initial concentration $u_0$ was 
bounded from above by $3.5$, 
values of more than $150$ are attained. Therefore, it seems reasonable to expect results resembling those dealing with transient growth phenomena for single-species chemotaxis models \cite{winkler_ctexceed,lankeit_exceed} also in the present two-species context.

\begin{figure}
\begin{center}\includegraphics{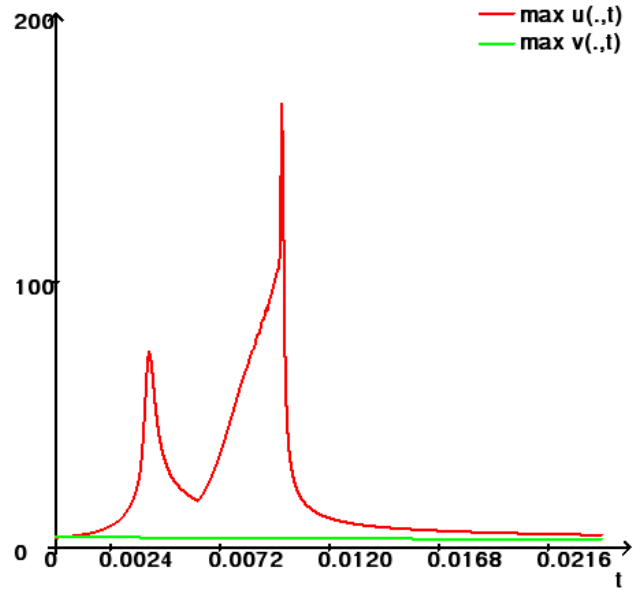}\end{center}
\caption{Evolution of the maximal population densities}\label{fig-max}
\end{figure}

Naturally, up to now analytical results (including those of the present article) have been confined to the solution behaviour in the limit $t\to \infty$, so that rigorous description and understanding of the intriguing performance of solutions over short periods of time remain as possibly worthwhile, albeit challenging, questions for future studies. 

\FloatBarrier
%==========================================================
%%%%%%%                                             %%%%%%%
  %%%                                                 %%%
 %%%                                                   %%%
%%%                    Acknowledgments                  %%%
 %%%                                                   %%%
  %%%                                                 %%%
%%%%%%%                                             %%%%%%%
%==========================================================
%\smallskip
\section*{Acknowledgements}
A major part of this work was written while M.M. visited Paderborn University in February 2016 and during the ``International Workshop on Mathematical Analysis of Chemotaxis'' in Tokyo, in which all authors participated. The authors are grateful to Tokyo University of Science for funding these events. 

%%%%%%%%%%%%%%%%%%%%%%%%%%%%%%%%%%%%%%%%%%
%         thebibliography                %
%%%%%%%%%%%%%%%%%%%%%%%%%%%%%%%%%%%%%%%%%%

\newpage
 {\scriptsize%\footnotesize 
 %\bibliographystyle{abbrv}
 %\bibliography{lit_twospecies}

\begin{thebibliography}{10}

 \setlength{\parskip}{0pt}
 \setlength{\itemsep}{0pt}

\bibitem{aida_tsujikawa_efendiev_yagi_mimura_06}
M.~Aida, T.~Tsujikawa, M.~Efendiev, A.~Yagi, and M.~Mimura.
\newblock Lower estimate of the attractor dimension for a chemotaxis growth
  system.
\newblock {\em Journal of the London Mathematical Society}, 74(2):453--474,
  2006.

\bibitem{bai_winkler}
X.~Bai and M.~Winkler.
\newblock Equilibration in a fully parabolic two-species chemotaxis system with
  competitive kinetics.
\newblock 2016.
\newblock preprint.

\bibitem{BBTW}
N.~Bellomo, A.~Bellouquid, Y.~Tao, and M.~Winkler.
\newblock Toward a mathematical theory of {K}eller-{S}egel models of pattern
  formation in biological tissues.
\newblock {\em Math. Models Methods Appl. Sci.}, 25(9):1663--1763, 2015.

\bibitem{bendahmane_langlais}
M.~Bendahmane and M.~Langlais.
\newblock A reaction-diffusion system with cross-diffusion modeling the spread
  of an epidemic disease.
\newblock {\em J. Evol. Equ.}, 10(4):883--904, 2010.

\bibitem{biler_espejo_guerra}
P.~Biler, E.~E. Espejo, and I.~Guerra.
\newblock Blowup in higher dimensional two species chemotactic systems.
\newblock {\em Commun. Pure Appl. Anal.}, 12(1):89--98, 2013.

\bibitem{biler_guerra}
P.~Biler and I.~Guerra.
\newblock Blowup and self-similar solutions for two-component drift-diffusion
  systems.
\newblock {\em Nonlinear Anal.}, 75(13):5186--5193, 2012.

\bibitem{cantrell_cosner_lou_advectionmediated}
R.~S. Cantrell, C.~Cosner, and Y.~Lou.
\newblock Advection-mediated coexistence of competing species.
\newblock {\em Proc. Roy. Soc. Edinburgh Sect. A}, 137(3):497--518, 2007.

\bibitem{chaplain_anderson}
M.~A. Chaplain and A.~R. Anderson.
\newblock Mathematical modelling of tissue invasion.
\newblock {\em Cancer modelling and simulation}, pages 269--297, 2003.

\bibitem{dickstein}
F.~Dickstein.
\newblock Sharp conditions for blowup of solutions of a chemotactical model for
  two species in {$\mathbb{R}^2$}.
\newblock {\em J. Math. Anal. Appl.}, 397(2):441--453, 2013.

\bibitem{espejo_vilches_conca}
E.~Espejo, K.~Vilches, and C.~Conca.
\newblock Sharp condition for blow-up and global existence in a two species
  chemotactic {K}eller-{S}egel system in {$\mathbb{R}^2$}.
\newblock {\em European J. Appl. Math.}, 24(2):297--313, 2013.

\bibitem{espejo_stevens_suzuki}
E.~E. Espejo, A.~Stevens, and T.~Suzuki.
\newblock Simultaneous blowup and mass separation during collapse in an
  interacting system of chemotactic species.
\newblock {\em Differential Integral Equations}, 25(3-4):251--288, 2012.

\bibitem{espejo_stevens_velazquez_nonsimultaneous}
E.~E. Espejo, A.~Stevens, and J.~J.~L. Vel{\'a}zquez.
\newblock A note on non-simultaneous blow-up for a drift-diffusion model.
\newblock {\em Differential Integral Equations}, 23(5-6):451--462, 2010.

\bibitem{espejo_stevens_velazquez_simultaneous}
E.~E. Espejo~Arenas, A.~Stevens, and J.~J.~L. Vel{\'a}zquez.
\newblock Simultaneous finite time blow-up in a two-species model for
  chemotaxis.
\newblock {\em Analysis (Munich)}, 29(3):317--338, 2009.

\bibitem{goh_twospecies}
B.~S. Goh.
\newblock Global stability in two species interactions.
\newblock {\em J. Math. Biol.}, 3(3-4):313--318, 1976.

\bibitem{hillen_painter_09}
T.~Hillen and K.~J. Painter.
\newblock A user's guide to {PDE} models for chemotaxis.
\newblock {\em J. Math. Biol.}, 58(1-2):183--217, 2009.

\bibitem{horstmann_03}
D.~Horstmann.
\newblock From 1970 until present: the {K}eller-{S}egel model in chemotaxis and
  its consequences. {I}.
\newblock {\em Jahresber. Deutsch. Math.-Verein.}, 105(3):103--165, 2003.

\bibitem{horstmann_generalizingKS}
D.~Horstmann.
\newblock Generalizing the {K}eller-{S}egel model: {L}yapunov functionals,
  steady state analysis, and blow-up results for multi-species chemotaxis
  models in the presence of attraction and repulsion between competitive
  interacting species.
\newblock {\em J. Nonlinear Sci.}, 21(2):231--270, 2011.

\bibitem{kurganov_lukacovamedvidova}
A.~Kurganov and M.~Luk{\'a}{\v{c}}ov{\'a}-Medvi{\v{d}}ov{\'a}.
\newblock Numerical study of two-species chemotaxis models.
\newblock {\em Discrete Contin. Dyn. Syst. Ser. B}, 19(1):131--152, 2014.

\bibitem{kurokiba}
M.~Kurokiba.
\newblock Existence and blowing up for a system of the drift-diffusion equation
  in {$R^2$}.
\newblock {\em Differential Integral Equations}, 27(5-6):425--446, 2014.

\bibitem{kuto_osaki_sakurai_tsujikawa_12}
K.~Kuto, K.~Osaki, T.~Sakurai, and T.~Tsujikawa.
\newblock Spatial pattern formation in a chemotaxis-diffusion-growth model.
\newblock {\em Physica D: Nonlinear Phenomena}, 241(19):1629 -- 1639, 2012.

\bibitem{lankeit_exceed}
J.~Lankeit.
\newblock Chemotaxis can prevent thresholds on population density.
\newblock {\em Discrete and Continuous Dynamical Systems - Series B},
  20(5):1499--1527, 2015.

\bibitem{lankeit_evsmoothness}
J.~Lankeit.
\newblock Eventual smoothness and asymptotics in a three-dimensional chemotaxis
  system with logistic source.
\newblock {\em Journal of Differential Equations}, 258(4):1158 -- 1191, 2015.

\bibitem{lauffenburger_aris_keller}
D.~Lauffenburger, R.~Aris, and K.~Keller.
\newblock Effects of cell motility and chemotaxis on microbial population
  growth.
\newblock {\em Biophysical journal}, 40(3):209, 1982.

\bibitem{lauffenburger_rivero_kelly_ford_dirienzo}
D.~A. Lauffenburger, M.~Rivero, F.~Kelly, R.~Ford, and J.~DiRienzo.
\newblock Bacterial chemotaxis.
\newblock {\em Annals of the New York Academy of Sciences}, 506(1):281--295,
  1987.

\bibitem{yan}
Y.~Li.
\newblock Global bounded solutions and their asymptotic properties under small
  initial data condition in a two-dimensional chemotaxis system for two
  species.
\newblock {\em J. Math. Anal. Appl.}, 429(2):1291--1304, 2015.

\bibitem{lili}
Y.~Li and Y.~Li.
\newblock Finite-time blow-up in higher dimensional fully-parabolic chemotaxis
  system for two species.
\newblock {\em Nonlinear Anal.}, 109:72--84, 2014.

\bibitem{liu_ou}
J.~Liu and C.~Ou.
\newblock How many consumer levels can survive in a chemotactic food chain?
\newblock {\em Front. Math. China}, 4(3):495--521, 2009.

\bibitem{lou_tao_win}
Y.~Lou, Y.~Tao, and M.~Winkler.
\newblock Approaching the ideal free distribution in two-species competition
  models with fitness-dependent dispersal.
\newblock {\em SIAM J. Math. Anal.}, 46(2):1228--1262, 2014.

\bibitem{matsukuma_durston}
S.~Matsukuma and A.~Durston.
\newblock Chemotactic cell sorting in dictyostelium discoideum.
\newblock {\em Development}, 50(1):243--251, 1979.

\bibitem{masaaki_tomomi}
M.~Mizukami and T.~Yokota.
\newblock Global existence and asymptotic stability of solutions to a
  two-species chemotaxis system with any chemical diffusion.
\newblock 2016.
\newblock preprint.

\bibitem{murrayII}
J.~D. Murray.
\newblock {\em Mathematical Biology. II Spatial Models and Biomedical
  Applications $\{$Interdisciplinary Applied Mathematics V. 18$\}$}.
\newblock Springer-Verlag New York Incorporated, 2001.

\bibitem{nakaguchi_efendiev_08}
E.~Nakaguchi and M.~Efendiev.
\newblock On a new dimension estimate of the global attractor for
  chemotaxis-growth systems.
\newblock {\em Osaka Journal of Mathematics}, 45(2):273--281, 06 2008.

\bibitem{nakaguchi_osaki_13}
E.~Nakaguchi and K.~Osaki.
\newblock Global solutions and exponential attractors of a parabolic-parabolic
  system for chemotaxis with subquadratic degradation.
\newblock {\em Discrete Contin. Dyn. Syst. Ser. B}, 18(10):2627--2646, 2013.

\bibitem{negreanu_tello_comparison}
M.~Negreanu and J.~I. Tello.
\newblock On a comparison method to reaction-diffusion systems and its
  applications to chemotaxis.
\newblock {\em Discrete Contin. Dyn. Syst. Ser. B}, 18(10):2669--2688, 2013.

\bibitem{negreanu_tello_ppe}
M.~Negreanu and J.~I. Tello.
\newblock On a two species chemotaxis model with slow chemical diffusion.
\newblock {\em SIAM J. Math. Anal.}, 46(6):3761--3781, 2014.

\bibitem{negreanu_tello_ppo}
M.~Negreanu and J.~I. Tello.
\newblock Asymptotic stability of a two species chemotaxis system with
  non-diffusive chemoattractant.
\newblock {\em J. Differential Equations}, 258(5):1592--1617, 2015.

\bibitem{painter_hillen_physicaD}
K.~J. Painter and T.~Hillen.
\newblock Spatio-temporal chaos in a chemotaxis model.
\newblock {\em Physica D: Nonlinear Phenomena}, 240(4-5):363--375, 2011.

\bibitem{pearce_etal}
I.~G. Pearce, M.~A. Chaplain, P.~G. Schofield, A.~R. Anderson, and S.~F.
  Hubbard.
\newblock Chemotaxis-induced spatio-temporal heterogeneity in multi-species
  host-parasitoid systems.
\newblock {\em Journal of mathematical biology}, 55(3):365--388, 2007.

\bibitem{stinner_tello_winkler}
C.~Stinner, J.~I. Tello, and M.~Winkler.
\newblock Competitive exclusion in a two-species chemotaxis model.
\newblock {\em J. Math. Biol.}, 68(7):1607--1626, 2014.

\bibitem{tang_tao}
X.~Tang and Y.~Tao.
\newblock Analysis of a chemotaxis model for multi-species host-parasitoid
  interactions.
\newblock {\em Appl. Math. Sci. (Ruse)}, 2(25-28):1239--1252, 2008.

\bibitem{taowin_twospecies_twochemicals}
Y.~Tao and M.~Winkler.
\newblock Boundedness vs. blow-up in a two-species chemotaxis system with two
  chemicals.
\newblock {\em Discrete Contin. Dyn. Syst. Ser. B}, 20(9):3165--3183, 2015.

\bibitem{tello_winkler_logsource}
J.~I. Tello and M.~Winkler.
\newblock A chemotaxis system with logistic source.
\newblock {\em Comm. Partial Differential Equations}, 32(4-6):849--877, 2007.

\bibitem{tello_winkler}
J.~I. Tello and M.~Winkler.
\newblock Stabilization in a two-species chemotaxis system with a logistic
  source.
\newblock {\em Nonlinearity}, 25(5):1413--1425, 2012.

\bibitem{vasiev_weijer}
B.~Vasiev and C.~J. Weijer.
\newblock Modeling chemotactic cell sorting during dictyostelium discoideum
  mound formation.
\newblock {\em Biophysical journal}, 76(2):595--605, 1999.

\bibitem{wang_yang_zhang}
Q.~{Wang}, J.~{Yang}, and L.~{Zhang}.
\newblock {Time periodic and stable patterns of a two--competing--species
  Keller--Segel chemotaxis model effect of cellular growth}.
\newblock {\em ArXiv e-prints}, May 2015.

\bibitem{wang_zhang_yang_hu}
Q.~Wang, L.~Zhang, J.~Yang, and J.~Hu.
\newblock Global existence and steady states of a two competing species
  {K}eller-{S}egel chemotaxis model.
\newblock {\em Kinet. Relat. Models}, 8(4):777--807, 2015.

\bibitem{wang_li}
W.~Wang and Y.~Li.
\newblock Stabilization in an {$n$}-species chemotaxis system with a logistic
  source.
\newblock {\em J. Math. Anal. Appl.}, 432(1):274--288, 2015.

\bibitem{wang_wu}
X.~Wang and Y.~Wu.
\newblock Qualitative analysis on a chemotactic diffusion model for two species
  competing for a limited resource.
\newblock {\em Quart. Appl. Math.}, 60(3):505--531, 2002.

\bibitem{winkler_08}
M.~Winkler.
\newblock Chemotaxis with logistic source: very weak global solutions and their
  boundedness properties.
\newblock {\em J. Math. Anal. Appl.}, 348(2):708--729, 2008.

\bibitem{winkler_parabolic_logsource}
M.~Winkler.
\newblock Boundedness in the higher-dimensional parabolic-parabolic chemotaxis
  system with logistic source.
\newblock {\em Comm. Partial Differential Equations}, 35(8):1516--1537, 2010.

\bibitem{winkler_ctexceed}
M.~Winkler.
\newblock How far can chemotactic cross-diffusion enforce exceeding carrying
  capacities?
\newblock {\em Journal of Nonlinear Science}, pages 1--47, 2014.

\bibitem{wolansky}
G.~Wolansky.
\newblock Multi-components chemotactic system in the absence of conflicts.
\newblock {\em European J. Appl. Math.}, 13(6):641--661, 2002.

\bibitem{zhang_li}
Q.~Zhang and Y.~Li.
\newblock Global existence and asymptotic properties of the solution to a
  two-species chemotaxis system.
\newblock {\em J. Math. Anal. Appl.}, 418(1):47--63, 2014.

\bibitem{zhang_li_globalboundedness_twospecies}
Q.~Zhang and Y.~Li.
\newblock Global boundedness of solutions to a two-species chemotaxis system.
\newblock {\em Z. Angew. Math. Phys.}, 66(1):83--93, 2015.

\bibitem{zhangzhenbu}
Z.~Zhang.
\newblock Existence of global solution and nontrivial steady states for a
  system modeling chemotaxis.
\newblock In {\em Abstract and Applied Analysis}, volume 2006. Hindawi
  Publishing Corporation, 2006.

\end{thebibliography}

 }
\end{document}